\RequirePackage{fix-cm}
\documentclass[smallextended]{svjour3}       
\smartqed  
\usepackage{graphicx}
\usepackage{amsfonts}
\usepackage{amsmath}
\usepackage{mathtools}
\def\dprod{\mathop{\displaystyle \prod }}
\begin{document}

\title{The Hermite-Hadamard inequality revisited: Some new proofs and
	applications
}


\author{Ilham A. Aliev         \and
	Mehmet E. Tamar  \and
	Cagla Sekin 
}

\institute{I. A. ALIEV \at
	Department of mathematics, Akdeniz University, Antalya, TURKEY \\
	\email{ialiev@akdeniz.edu.tr}           
	\and
	M. E. TAMAR \at
	Institute of science, Akdeniz University, Antalya, TURKEY\\
	\and
	C. SEKIN \at
	Institute of science, Akdeniz University, Antalya, TURKEY\\
}
\date{Received: date / Accepted: date}

\maketitle

\begin{abstract}
	New proofs of the classical Hermite-Hadamard inequality are presented and
	several applications are given, including Hadamard-type inequalities for the
	functions, whose derivatives have inflection points or whose derivatives are convex. Morever, 
	some estimates from below and above for the  first moments of functions $%
	f:[a,b]\rightarrow 
	\mathbb{R}
	$ about the center point $c=(a+b)/2$ are obtained and the reverse Hardy
	inequality for convex functions $f:[0,\infty )\rightarrow (0,\infty )$ is
	established.
	
	\keywords{Convex functions \and  Hermite-Hadamard inequality \and Jensen inequality \and Fejer inequality. }
	 \subclass{26D15 \and 26A51 \and 26D10}
\end{abstract}

\section{Introduction}

The famous Hermite-Hadamard inequality (HH) asserts that the integral mean value of a convex function $f:[a,b]\rightarrow 
\mathbb{R}
$ can be estimated above and belove by its values at the points $a,b$ and $%
(a+b)/2$. More precisely,%
\begin{equation}
	f\left( \frac{a+b}{2}\right) \leq \frac{1}{b-a}\int_{a}^{b}f(x)dx\leq \frac{%
		f(a)+f(b)}{2}.  \tag{HH}  \label{HH}
\end{equation}

Equality holds only for functions of the form $f(x)=cx+d$. According to the
notations by Niculescu and Persson \cite{10}, the right and left parts of
(HH) we denote by (RHH) and (LHH), respectively. Some authors also refer to
the (RHH) as Hadamard's inequality.

(HH) has many generalizations, extensions, refinements and there is a large
number of papers and book's chapters in this area; see, e.g. books by
Niculescu and Persson \cite{11}; Mitrinovic, Pecaric and Fink \cite{23};
Dragomir and Pearce \cite{6} and papers \cite{22,27,1,2,3,28,4,5,7,8,12,9,19,20,21,14,25,10,13,26,29,16,17,18}, which are a small part of the relevant references.

The plan of this article is as follows.

In section 2 we give two new proofs of (HH), one of which is extremely short
(so to speak, "without pulling the pen out of paper"), and the other is
based on Riemann's integral sums. As an application, we give an estimation
from below and above of the integral of the convex function $f:[0,\infty
)\rightarrow (0,\infty )$ via the series $\overset{\infty }{\underset{1}{%
		\sum }}f(k)$ and $\overset{\infty }{\underset{1}{\sum }}f\left( k-\frac{1}{2}%
\right) .$

In section 3, we give some new inequalities arising as a combination of the
(HH) with the Hardy inequality and some variant of H\"{o}lder's inequality.
For example, as a consequence we prove that, if $f:[0,\infty )\rightarrow
(0,\infty )$ is convex and $f\in L_{p}(0,\infty )$, $\forall p>1$, then%
\begin{equation*}
	\underset{p\rightarrow \infty }{\lim }\frac{\left\Vert \frac{1}{x}%
		\int_{0}^{x}f\right\Vert _{p}}{\left\Vert f\right\Vert _{p}}=1.
\end{equation*}

Section 4 is devoted to Hadamard's type inequality, i.e. (RHH) for the
functions whose first derivatives have an inflection point. As a particular
case, we show that if $f^{\prime }$ is concave on $\left[ a,\frac{a+b}{2}\right] $ and
convex on $\left[ \frac{a+b}{2},b\right] $, then%
\begin{equation*}
	\frac{f(a)+f(b)}{2}-\frac{1}{b-a}\int_{a}^{b}f(x)dx\leq \frac{b-a}{12}\left(
	f^{\prime }(b)-f^{\prime }(a)\right) .
\end{equation*}

In the last 5th section we prove various inequalities for functions having
convex first or second order derivatives. As far as we know, in the
literature on this theme there are some inequalities for functions whose 
\textit{absolute values} of the derivatives are convex, see, e.g. \cite{27,28,29}. In our Theorem \ref{teo5.1}, \ref{teo5.2}, and \ref{teo5.3} in
section 5, the convexity condition is imposed on the derivatives themselves,
but not on their absolute values. One of the interesting particular results
obtained in this section is as follows.

Given $f:[a,b]\rightarrow 
\mathbb{R}
$, let $f^{\prime}$ be convex. Then%
\begin{equation*}
	\int_{\frac{a+b}{2}}^{b}f(x)dx-\int_{a}^{\frac{a+b}{2}}f(x)dx\leq \frac{b-a}{%
		4}(f(b)-f(a)).
\end{equation*}

Another result in this section is the estimation from above and below of the
first moment of a function $f:[a,b]\rightarrow 
\mathbb{R}
$ about the center point $c=(a+b)/2$, i.e. the integral $M_{f}=\int_{a}^{b}%
\left( x-\frac{a+b}{2}\right) f(x)dx$, when $f^{\prime }$ is convex.

\section{Two new proof of (HH) and some applications}

At first, we give an auxiliary inequality that is satisfied by convex
functions.

\begin{lemma}
	(cf. Lemma 1.3 in \cite{25}) Let $f$ be a convex function on $[a,b]$. Then%
	\begin{equation}
		f(a)+f(b)\geq f(a+b-x)+f(x)\text{, (}\forall x\in \lbrack a,b]\text{).}
		\label{2.1}
	\end{equation}
\end{lemma}

\begin{proof}
	Let $x\in \lbrack a,b]$. There exists a $t\in \lbrack 0,1]$ such that $%
	x=ta+(1-t)b$. Then $a+b-x=(1-t)a+tb$ and therefore, 
	\begin{eqnarray*}
		f(a+b-x) &=&f((1-t)a+tb)\leq (1-t)f(a)+tf(b) \\
		&=&f(a)+f(b)-[tf(a)+(1-t)f(b)] \\
		&\leq &f(a)+f(b)-f(ta+(1-t)b) \\
		&=&f(a)+f(b)-f(x).
	\end{eqnarray*}
\end{proof}

By making use of (\ref{2.1}) we give here a short proof of the (HH),
"without pulling the pen out of paper".

\begin{proof} \textit{of (HH)} 
	Integrating the inequality $f(a)+f(b)\geq f(a+b-x)+f(x)$ over $[a,b]$ and
	using 
	\begin{equation*}
		\int_{a}^{b}f(a+b-x)dx=\int_{a}^{b}f(x)dx
	\end{equation*}%
	we have 
	\begin{eqnarray*}
		(b-a)(f(a)+f(b)) &\geq &2\int_{a}^{b}f(x)dx=\int_{a}^{b}(f(a+b-x)+f(x))dx \\
		&=&2\int_{a}^{b}\frac{f(a+b-x)+f(x)}{2}dx\geq 2\int_{a}^{b}f\left( \frac{%
			a+b-x+x}{2}\right) dx \\
		&=&2(b-a)f\left( \frac{a+b}{2}\right) .
	\end{eqnarray*}
\end{proof}

\begin{remark}
	Although the (HH) has several proofs, as far as we know the first simple
	proof was given by Azbetia \cite{1}; (see, also Niculescu and Persson \cite%
	{10}, p. 664). Another simple proof and refinement was given by El Farissi 
	\cite{8}.
\end{remark}

\begin{remark}
	The inequality (\ref{2.1}) enables one also to give a short proof of
	"Hadamard part" of the Fejer inequality:
	
	If $g\geq 0$ is integrable and symmetric with respect to $\frac{a+b}{2}$,
	i.e. $g(a+b-x)=g(x)$, ($x\in \lbrack a,b]$), we have%
	\begin{eqnarray*}
		\int_{a}^{b}f(x)g(x)dx &=&\frac{1}{2}\left[ \int_{a}^{b}f(x)g(x)dx+%
		\int_{a}^{b}f(a+b-x)g(a+b-x)dx\right] \\
		&=&\frac{1}{2}\int_{a}^{b}(f(x)+f(a+b-x))g(x)dx\overset{(\ref{2.1})}{\leq }%
		\frac{f(a)+f(b)}{2}\int_{a}^{b}g(x)dx.
	\end{eqnarray*}
\end{remark}

The following theorem is the "Riemann integral's sums version" of (HH).

\begin{theorem}
	If $f:[a,b]\rightarrow 
	\mathbb{R}
	$ is convex and $x_{k}=a+k\frac{b-a}{n}$, ($k=1,2,\cdots ,n$) then for any $%
	n\in 
	\mathbb{N}
	$ we have 
	\begin{equation}
		f\left( \frac{\left( 1-\frac{1}{n}\right) a+\left( 1+\frac{1}{n}\right) b}{2}%
		\right) \leq \frac{1}{n}\overset{n}{\underset{k=1}{\sum }}f(x_{k})\leq \frac{%
			1}{2}\left[ f(a)\left( 1-\frac{1}{n}\right) +f(b)\left( 1+\frac{1}{n}\right) %
		\right] .  \label{2.2}
	\end{equation}
\end{theorem}

\begin{proof}
	Let $x_{k}=a+k\frac{b-a}{n}$, ($k=1,2,\cdots ,n$). Then writing $x_{k}$ as 
	\begin{equation*}
		x_{k}=\frac{b-x_{k}}{b-a}a+\frac{x_{k}-a}{b-a}b
	\end{equation*}%
	and using%
	\begin{equation*}
		f(x_{k})\leq \frac{b-x_{k}}{b-a}f(a)+\frac{x_{k}-a}{b-a}f(b),
	\end{equation*}%
	one has%
	\begin{eqnarray*}
		\overset{n}{\underset{k=1}{\sum }}f(x_{k}) &\leq &\frac{f(a)}{b-a}\overset{n}%
		{\underset{k=1}{\sum }}(b-x_{k})+\frac{f(b)}{b-a}\overset{n}{\underset{k=1}{%
				\sum }}(x_{k}-a) \\
		&=&\frac{1}{2}[f(a)(n-1)+f(b)(n+1)],
	\end{eqnarray*}%
	and therefore,%
	\begin{equation}
		\frac{1}{n}\overset{n}{\underset{k=1}{\sum }}f(x_{k})\leq \frac{1}{2}\left[
		f(a)\left( 1-\frac{1}{n}\right) +f(b)\left( 1+\frac{1}{n}\right) \right] .
		\label{2.3}
	\end{equation}%
	On the other hand, the Jensen inequality yields%
	\begin{equation}
		\frac{1}{n}\overset{n}{\underset{k=1}{\sum }}f(x_{k})\geq f\left( \frac{1}{n}%
		\overset{n}{\underset{k=1}{\sum }}x_{k}\right) =f\left( \frac{\left( 1-\frac{%
				1}{n}\right) a+\left( 1+\frac{1}{n}\right) b}{2}\right) .  \label{2.4}
	\end{equation}
	
	By combining (\ref{2.3}) and (\ref{2.4}) we obtain (\ref{2.2}).
\end{proof}

\begin{corollary}
	After taking limit as $n\rightarrow \infty $ in (\ref{2.2}) and using the
	fact that the convex function is continuous, we obtain (HH).
\end{corollary}

The following two theorems are the simple consequences of (HH).

\begin{theorem}[a refinement of (RHH)]
	Let $f:[a,b]\rightarrow 
	\mathbb{R}
	$ be convex. Then%
	\begin{eqnarray}
		\frac{1}{b-a}\int_{a}^{b}f(x)dx &\leq &\frac{1}{b-a}\int_{a}^{b}f(x)\left[
		\ln \frac{(b-a)^{2}}{(b-x)(x-a)}-1\right] dx  \notag \\
		&\leq &\frac{f(a)+f(b)}{2}  \label{2.5}
	\end{eqnarray}
\end{theorem}

\begin{proof}
	For any $x\in (a,b]$ one has%
	\begin{equation*}
		f\left( \frac{a+x}{2}\right) \leq \frac{1}{x-a}\int_{a}^{x}f(t)dt\leq \frac{%
			f(a)+f(x)}{2}.
	\end{equation*}%
	Integrating over $(a,b)$ we have 
	\begin{equation}
		\int_{a}^{b}f\left( \frac{a+x}{2}\right) dx\leq \int_{a}^{b}\frac{1}{x-a}%
		\left( \int_{a}^{x}f(t)dt\right) dx\leq \int_{a}^{b}\frac{f(a)+f(x)}{2}dx.
		\label{2.6}
	\end{equation}%
	After simple calculations, (\ref{2.6}) leads to%
	\begin{equation}
		2\int_{a}^{\frac{a+b}{2}}f(x)dx\leq \int_{a}^{b}f(x)\ln \frac{b-a}{x-a}%
		dx\leq \frac{1}{2}\left[ f(a)(b-a)+\int_{a}^{b}f(x)dx\right] .  \label{2.7}
	\end{equation}%
	Similarly, integrating the inequality%
	\begin{equation*}
		f\left( \frac{x+b}{2}\right) \leq \frac{1}{b-x}\int_{a}^{b}f(t)dt\leq \frac{%
			f(x)+f(b)}{2}
	\end{equation*}%
	over $(a,b)$ we get 
	\begin{equation*}
		\int_{a}^{b}f\left( \frac{x+b}{2}\right) dx\leq \int_{a}^{b}\frac{1}{b-x}%
		\left( \int_{a}^{b}f(t)dt\right) dx\leq \int_{a}^{b}\frac{f(x)+f(b)}{2}dx
	\end{equation*}%
	which leads to%
	\begin{equation}
		2\int_{\frac{a+b}{2}}^{b}f(x)dx\leq \int_{a}^{b}f(x)\ln \frac{b-a}{b-x}%
		dx\leq \frac{1}{2}\left[ f(b)(b-a)+\int_{a}^{b}f(x)dx\right] .  \label{2.8}
	\end{equation}%
	After summing up (\ref{2.7}) and (\ref{2.8}) we obtain (\ref{2.5}).
\end{proof}

\begin{theorem}
	Let $f:[0,\infty )\rightarrow (0,\infty )$ be a strictly convex function and 
	$\overset{\infty }{\underset{k=1}{\sum }}f(x_{k})<\infty $. Then%
	\begin{equation}
		\overset{\infty }{\underset{k=1}{\sum }}f\left( k-\frac{1}{2}\right)
		<\int_{a}^{b}f(x)dx<\frac{1}{2}f(0)+\overset{\infty }{\underset{k=1}{\sum }}%
		f(k).  \label{2.9}
	\end{equation}
\end{theorem}

\begin{proof}
	Denote $x_{0}=a$ and $x_{k}=a+(k-1)\frac{b-a}{n}$, ($k=1,2,\cdots ,n$).
	Since $f$ is strictly convex, we have%
	\begin{equation*}
		f\left( \frac{x_{k-1}+x_{k}}{2}\right) <\frac{1}{x_{k}-x_{k-1}}%
		\int_{x_{k-1}}^{x_{k}}f(x)dx<\frac{f(x_{k-1})+f(x_{k})}{2}\text{, }%
		(k=1,2,\cdots ,n).
	\end{equation*}%
	Taking into account the formulas%
	\begin{equation*}
		x_{k}-x_{k-1}=\frac{b-a}{n}\text{, }\frac{x_{k-1}+x_{k}}{2}=a+\left( k-\frac{%
			1}{2}\right) \frac{b-a}{n}
	\end{equation*}%
	and summing the inequalities above we obtain%
	\begin{eqnarray*}
		\overset{n}{\underset{k=1}{\sum }}\frac{1}{n}f\left( a+\left( k-\frac{1}{2}%
		\right) \frac{b-a}{n}\right) &<&\frac{1}{b-a}\int_{a}^{b}f(x)dx \\
		&<&\frac{1}{n}\left[ \frac{f(a)+f(b)}{2}+\overset{n-1}{\underset{k=1}{\sum }}%
		f\left( a+k\frac{b-a}{n}\right) \right]
	\end{eqnarray*}%
	Further, setting $a=0$, $b=n$ we have 
	\begin{equation*}
		\overset{n}{\underset{k=1}{\sum }}f\left( k-\frac{1}{2}\right)
		<\int_{a}^{b}f(x)dx<\frac{f(0)+f(n)}{2}+\overset{n-1}{\underset{k=1}{\sum }}%
		f(k).
	\end{equation*}%
	Taking limit as $n\rightarrow \infty $ and using $\underset{n\rightarrow
		\infty }{\lim }f(n)=0$ we obtain the desired formula (\ref{2.9}).
\end{proof}

\begin{remark}
	Since $f:[0,\infty )\rightarrow (0,\infty )$ is convex and $\underset{%
		n\rightarrow \infty }{\lim }f(n)$ is finite (actually, zero), then $f$ is
	monotonically decreasing and therefore the comparison of the areas under
	graphics gives%
	\begin{equation}
		\overset{\infty }{\underset{k=1}{\sum }}f(k)<\int_{0}^{\infty }f(x)dx<f(0)+%
		\overset{\infty }{\underset{k=1}{\sum }}f(k).  \label{2.10}
	\end{equation}%
	It is clear that, the inequalities (\ref{2.9}) are better than (\ref{2.10}).
\end{remark}

\begin{example}
	If $f(x)=e^{-x}$ then from (\ref{2.9}) we have 
	\begin{equation*}
		\frac{\sqrt{e}}{e-1}<1<\frac{1}{2}+\frac{1}{e-1}\Leftrightarrow \sqrt{e}<e-1<%
		\frac{1}{2}(e+1),
	\end{equation*}%
	whereas the formula (\ref{2.10}) gives the rougher estimate $1<e-1<e$.
\end{example}

\section{Some inequalities arising as a combination of the Hermite-Hadamard
	inequality with the other ones}

\begin{theorem}
	\label{teo3.1}Let $1<p<\infty $ and $\alpha p>1$. Let further, $f:[0,\infty
	)\rightarrow (0,\infty )$ be convex and such that%
	\begin{equation*}
		\left\Vert x^{1-\alpha }f(x)\right\Vert _{p}\equiv \left( \int_{0}^{\infty
		}\left( x^{1-\alpha }f(x)\right) ^{p}dx\right) ^{1/p}<\infty .
	\end{equation*}%
	Then%
	\begin{equation}
		2^{1-\alpha +\frac{1}{p}}\leq \frac{\left\Vert x^{-\alpha
			}\int_{0}^{x}f\right\Vert _{p}}{\left\Vert x^{1-\alpha }f(x)\right\Vert _{p}}%
		\leq \frac{1}{\alpha -1/p}.  \label{3.1}
	\end{equation}
\end{theorem}

\begin{corollary}
	(a) If $\alpha =1$, then%
	\begin{equation}
		2^{\frac{1}{p}}\leq \frac{\left\Vert \frac{1}{x}\int_{0}^{x}f\right\Vert _{p}%
		}{\left\Vert f\right\Vert _{p}}\leq \frac{1}{1-1/p}.  \label{3.2}
	\end{equation}
	
	(b) Let, in addition, $f\in L_{p}(0,\infty )$, ($\forall p>1$). Then by
	taking the limit in (\ref{3.2}) as $p\rightarrow \infty $ one has%
	\begin{equation}
		\underset{p\rightarrow \infty }{\lim }\frac{\left\Vert \frac{1}{x}%
			\int_{0}^{x}f\right\Vert _{p}}{\left\Vert f\right\Vert _{p}}=1.  \label{3.3}
	\end{equation}
\end{corollary}

\begin{proof}
	We will use the classical weighted Hardy inequality, which asserts that%
	\begin{equation}
		\left( \int_{0}^{\infty }\left\vert x^{-\alpha
		}\int_{0}^{x}f(t)dt\right\vert ^{p}dx\right) ^{1/p}\leq c\left(
		\int_{0}^{\infty }\left\vert x^{1-\alpha }f(x)\right\vert ^{p}dx\right)
		^{1/p},  \label{3.4}
	\end{equation}%
	where $c=\frac{p}{\alpha p-1}$, $1<p<\infty $, $\alpha p>1.$
	
	Now, by (LHH) we have 
	\begin{equation*}
		f\left( \frac{x}{2}\right) <\frac{1}{x}\int_{0}^{x}f(t)dt\Rightarrow
		x^{1-\alpha }f\left( \frac{x}{2}\right) <x^{-\alpha }\int_{0}^{x}f(t)dt,
	\end{equation*}%
	and therefore%
	\begin{eqnarray}
		\int_{0}^{\infty }\left( x^{1-\alpha }f\left( \frac{x}{2}\right) \right)
		^{p}dx &\leq &\int_{0}^{\infty }\left( x^{-\alpha }\int_{0}^{x}f(t)dt\right)
		^{p}dx  \notag \\
		&&\overset{(\ref{3.4})}{\leq }\left( \frac{p}{\alpha p-1}\right)
		^{p}\int_{0}^{\infty }(x^{1-\alpha }f(x))^{p}dx.  \label{3.5}
	\end{eqnarray}%
	Since%
	\begin{equation*}
		\int_{0}^{\infty }\left( x^{1-\alpha }f\left( \frac{x}{2}\right) \right)
		^{p}dx=2^{p(1-\alpha )+1}\int_{0}^{\infty }\left( x^{1-\alpha }f(x)\right)
		^{p}dx,
	\end{equation*}%
	we have from (\ref{3.5}) the desired result (\ref{3.1}) and its consequences
	(\ref{3.2}) and (\ref{3.3}).
\end{proof}

\begin{remark}
	The left hand side of (\ref{3.1}) shows that under the conditions of Theorem %
	\ref{teo3.1} the following reverse Hardy's inequality is valid:%
	\begin{equation*}
		\left\Vert x^{-\alpha }\int_{0}^{x}f\right\Vert _{p}\geq 2^{1-\alpha +\frac{1%
			}{p}}\left\Vert x^{1-\alpha }f(x)\right\Vert _{p}.
	\end{equation*}
\end{remark}

\begin{example}
	Let $k>0$ and $f(x)=e^{-kx}$. Then (\ref{3.3}) yields 
	\begin{equation*}
		\underset{p\rightarrow \infty }{\lim }\left( \int_{0}^{\infty }\left( \frac{%
			1-e^{-kx}}{x}\right) ^{p}dx\right) ^{1/p}=k.
	\end{equation*}
\end{example}

In the next theorem we will make use of a combination of Hadamard's
inequality 
\begin{equation*}
	\frac{1}{b-a}\int_{a}^{b}f(x)dx\leq \frac{f(a)+f(b)}{2}
\end{equation*}%
and the inequality%
\begin{equation}
	\left( \int_{a}^{b}\left( \overset{n}{\underset{k=1}{\dprod }}%
	u_{k}(x)\right) dx\right) ^{n}\leq \overset{n}{\underset{k=1}{\dprod }}%
	\left( \int_{a}^{b}u_{k}^{n}(x)dx\right) ,  \label{3.6}
\end{equation}%
where $u_{1}\geq 0,\cdots ,u_{n}\geq 0.$

Recall that the inequality (\ref{3.6}) is a consequence of H\"{o}lder's
inequality and can be proved by induction.

We need also the following

\begin{lemma}
	\label{Lemma 3.2}If $u:[a,b]\rightarrow (0,\infty )$ is convex, then $u^{n}$
	is convex as well for any $n\in 
	\mathbb{N}
	.$
\end{lemma}

A simple proof follows by induction. Namely, if the functions $u>0$ and $%
u^{n}$ are convex for some $n\geq 2$, i.e.%
\begin{equation*}
	u(\alpha x+\beta y)\leq \alpha u(x)+\beta u(y)
\end{equation*}%
and%
\begin{equation*}
	u^{n}(\alpha x+\beta y)\leq \alpha u^{n}(x)+\beta u^{n}(y)\text{, }(\alpha
	+\beta =1),
\end{equation*}%
then by multiplying these inequalities we have%
\begin{eqnarray*}
	u^{n+1}(\alpha x+\beta y) &\leq &\left( \alpha u(x)+\beta u(y)\right) \left(
	\alpha u^{n}(x)+\beta u^{n}(y)\right) \\
	&\leq &\alpha u^{n+1}(x)+\beta u^{n+1}(y),
\end{eqnarray*}%
where the last estimate is equivalent to the following obvious inequality:%
\begin{equation*}
	(u(x)-u(y))\left( u^{n}(x)-u^{n}(y)\right) \geq 0.
\end{equation*}

\begin{remark}
	The convexity of the functions $u_{1}\geq 0,u_{2}\geq 0,\cdots ,u_{n}\geq 0$
	does not guarantee the convexity of their product $u_{1}u_{2}\cdots u_{n}$.
	Indeed, for example, although the functions $u_{1}(x)=x^{2},u_{2}(x)=x^{2},%
	\cdots ,u_{n-1}(x)=x^{2}$ and $u_{n}(x)=(2-x)^{2n-2}$, ($n\geq 2$) are
	convex on $[0,2]$, their product $u(x)=x^{2n-2}(2-x)^{2n-2}$ is not convex
	because of $f^{\prime \prime }(1)=4(2n-2)(1-n)<0.$
\end{remark}

\begin{theorem}
	For given $n\geq 2$, let the functions $u_{1}\geq 0,u_{2}\geq 0,\cdots
	,u_{n}\geq 0$ be convex on $[a,b]$. Then%
	\begin{equation}
		\frac{1}{b-a}\int_{a}^{b}\left( \overset{n}{\underset{k=1}{\dprod }}%
		u_{k}(x)\right) dx\leq \frac{1}{2}\overset{n}{\underset{k=1}{\dprod }}%
		(u_{k}^{n}(a)+u_{k}^{n}(b))^{\frac{1}{n}}.  \label{3.7}
	\end{equation}
\end{theorem}

\begin{proof}
	Since $u_{k}$, ($k=1,2,\cdots ,n$) is convex on $[a,b]$, then $u_{k}^{n}$ is
	also convex by Lemma \ref{Lemma 3.2}. Then Hadamard's inequality yields%
	\begin{equation*}
		\frac{1}{b-a}\int_{a}^{b}u_{k}^{n}(x)dx\leq \frac{1}{2}\left[
		u_{k}^{n}(a)+u_{k}^{n}(b)\right] ,\text{ }(k=1,2,\cdots ,n).
	\end{equation*}%
	By multiplying these inequalities we have%
	\begin{equation}
		\frac{1}{(b-a)^{n}}\overset{n}{\underset{k=1}{\dprod }}\left(
		\int_{a}^{b}u_{k}^{n}(x)dx\right) \leq \frac{1}{2^{n}}\overset{n}{\underset{%
				k=1}{\dprod }}\left( u_{k}^{n}(a)+u_{k}^{n}(b)\right) .  \label{3.8}
	\end{equation}%
	Here, by making use of the inequality (\ref{3.6}), we get 
	\begin{equation*}
		\frac{1}{(b-a)^{n}}\left( \int_{a}^{b}\left( \overset{n}{\underset{k=1}{%
				\dprod }}u_{k}(x)\right) dx\right) ^{n}\leq \frac{1}{2^{n}}\overset{n}{%
			\underset{k=1}{\dprod }}\left( u_{k}^{n}(a)+u_{k}^{n}(b)\right) ,
	\end{equation*}%
	from which the inequality (\ref{3.7}) follows.
\end{proof}

\begin{remark}
	For $n=2$, the inequality (\ref{3.7}) was proved by Amrahov \cite{22}.
	Another generalization of Amrahov's result for the product of two functions
	was noted by D. A. Ion \cite{21}:
	
	If $u\geq 0,v\geq 0$ are convex and $\frac{1}{p}+\frac{1}{q}=1$, ($%
	1<p,q<\infty $), then%
	\begin{equation*}
		\frac{1}{b-a}\int_{a}^{b}u(t)v(t)dt\leq \frac{1}{2}\left(
		u^{p}(a)+u^{p}(b)\right) ^{1/p}\left( u^{q}(a)+u^{q}(b)\right) ^{1/q}.
	\end{equation*}%
	It should also be mentioned that, in the same paper \cite{21}\ Ion gives
	some generalization of Amrahov's result for the product of two functions in
	Orlicz spaces.
\end{remark}

\section{Hadamard's type inequality for the functions whose derivatives have
	an inflection point}

\begin{theorem}
	Given $c\in \lbrack a,b]$ and $f:[a,b]\rightarrow 
	\mathbb{R}
	$, let its derivative $f^{\prime }$ be concave on $[a,c]$ and convex on $%
	[c,b]$. Then%
	\begin{eqnarray}
		&&\left[ \frac{c-a}{b-a}f(a)+\frac{b-c}{b-a}f(b)\right] -\frac{1}{b-a}%
		\int_{a}^{b}f(x)dx  \notag \\
		&\leq &\frac{1}{3}\left[ \frac{(b-c)^{2}}{b-a}f^{\prime }(b)-\frac{(c-a)^{2}%
		}{b-a}f^{\prime }(a)+\left( \frac{a+b}{2}-c\right) f^{\prime }(c)\right] .
		\label{4.1}
	\end{eqnarray}
\end{theorem}

\begin{corollary}
	In case of $c=\frac{a+b}{2}$ we have 
	\begin{equation}
		\frac{f(a)+f(b)}{2}-\frac{1}{b-a}\int_{a}^{b}f(x)dx\leq \frac{b-a}{12}%
		(f^{\prime }(b)-f^{\prime }(a))  \label{4.2}
	\end{equation}
\end{corollary}

\begin{proof}
	Integration by parts yields%
	\begin{equation*}
		\frac{c-a}{b-a}f(a)+\frac{b-c}{b-a}f(b)-\frac{1}{b-a}\int_{a}^{b}f(x)dx=%
		\frac{1}{b-a}\int_{a}^{b}(x-c)f^{\prime }(x)dx
	\end{equation*}%
	\begin{eqnarray}
		&=&\frac{1}{b-a}\int_{a}^{c}(x-c)f^{\prime }(x)dx+\frac{1}{b-a}%
		\int_{c}^{b}(x-c)f^{\prime }(x)dx  \notag \\
		&\equiv &A+B.  \label{4.3}
	\end{eqnarray}%
	By changing variables as $x=(1-\lambda )a+\lambda c$, ($0<\lambda <1$) in $A$
	and $x=(1-\lambda )c+\lambda b$ in $B$ and applying Jensen's inequality, we
	have%
	\begin{eqnarray}
		A &\equiv &\frac{1}{b-a}\int_{a}^{c}(x-c)f^{\prime }(x)dx=\frac{(a-c)^{2}}{%
			b-a}\int_{0}^{1}(\lambda -1)f^{\prime }((1-\lambda )a+\lambda c)d\lambda  
		\notag \\
		&\leq &\frac{(a-c)^{2}}{b-a}\int_{0}^{1}(\lambda -1)\left[ (1-\lambda
		)f^{\prime }(a)+\lambda f^{\prime }(c)\right] d\lambda   \notag \\
		&=&-\frac{(a-c)^{2}}{6\left( b-a\right) }\left[ 2f^{\prime }(a)+f^{\prime
		}(c)\right] ;  \label{4.4}
	\end{eqnarray}%
	\begin{eqnarray}
		B &\equiv &\frac{1}{b-a}\int_{c}^{b}(x-c)f^{\prime }(x)dx=\frac{(b-c)^{2}}{%
			(b-a)}\int_{0}^{1}\lambda f^{\prime }((1-\lambda )c+\lambda b)d\lambda  
		\notag \\
		&\leq &\frac{(b-c)^{2}}{(b-a)}\int_{0}^{1}\left( \lambda (1-\lambda
		)f^{\prime }(c)+\lambda ^{2}f^{\prime }(b)\right) d\lambda   \notag \\
		&=&\frac{(b-c)^{2}}{6(b-a)}[f^{\prime }(c)+2f^{\prime }(b)].  \label{4.5}
	\end{eqnarray}%
	It follows from (\ref{4.4}) and (\ref{4.5}) that%
	\begin{equation*}
		A+B\leq \frac{1}{3}f^{\prime }(b)\frac{(b-c)^{2}}{b-a}-\frac{1}{3}f^{\prime
		}(a)\frac{(a-c)^{2}}{b-a}+\frac{1}{6}f^{\prime }(c)(a+b-2c),
	\end{equation*}%
	which completes the proof.
\end{proof}

\begin{remark}
	A simple calculation shows that the equality in (\ref{4.1}) holds for the
	functions $f(x)=k(x-c)^{2}+m$, ($k,m\in 
	\mathbb{R}
	$).
\end{remark}

\begin{remark}
	In the "critic" cases $c=a$ or $c=b$, i.e. in the cases when $f^{\prime }$
	is convex or concave on $[a,b]$ we have from (\ref{4.1})%
	\begin{equation*}
		f(b)-\frac{1}{b-a}\int_{a}^{b}f(x)dx\leq \frac{b-a}{6}[f^{\prime
		}(a)+2f^{\prime }(b)]
	\end{equation*}%
	and%
	\begin{equation*}
		f(a)-\frac{1}{b-a}\int_{a}^{b}f(x)dx\leq -\frac{b-a}{6}[2f^{\prime
		}(a)+f^{\prime }(b)],
	\end{equation*}%
	respectively.
\end{remark}

\section{Various inequalities for functions having convex first or second
	order derivatives}

The first moment of a function $f$ about the center point $c=(a+b)/2$ is
defined by $M_{f}=\int_{a}^{b}\left( x-\frac{a+b}{2}\right) f(x)dx$. In the
following Theorem we obtain some estimation from above and below for $M_{f}$%
, when $f^{\prime }$ is convex.

\begin{theorem}
	\label{teo5.1}Suppose that the derivative $f^{\prime }$ of the function $%
	f:[a,b]\rightarrow 
	\mathbb{R}
	$ is convex. Then the first moment of $f$ about the center point $c=(a+b)/2$
	satisfies the following inequality 
	\begin{equation}
		A\leq \int_{a}^{b}\left( x-\frac{a+b}{2}\right) f(x)dx\leq B,  \label{5.1}
	\end{equation}%
	where%
	\begin{equation*}
		A=\frac{(a-b)^{2}}{8}(f(b)-f(a))-\frac{(b-a)^{3}}{48}(f^{\prime
		}(a)+f^{\prime }(b))
	\end{equation*}%
	and%
	\begin{equation*}
		B=\frac{(b-a)^{3}}{24}(f^{\prime }(a)+f^{\prime }(b)).
	\end{equation*}
\end{theorem}

\begin{proof}
	Integration by parts leads to%
	\begin{eqnarray*}
		\int_{a}^{b}(x-a)(b-x)f^{\prime }(x)dx &=&\int_{a}^{b}(x-a)(b-x)df(x) \\
		&=&2\int_{a}^{b}\left( x-\frac{a+b}{2}\right) f(x)dx.
	\end{eqnarray*}%
	Hence,%
	\begin{equation*}
		\int_{a}^{b}\left( x-\frac{a+b}{2}\right) f(x)dx=\frac{1}{2}%
		\int_{a}^{b}(x-a)(b-x)f^{\prime }(x)dx
	\end{equation*}%
	\begin{equation*}
		\text{(set }x=(1-t)a+tb\text{, }(x-a)(b-x)=(b-a)^{2}t(1-t)\text{, }0\leq
		t\leq 1\text{)}
	\end{equation*}%
	\begin{eqnarray*}
		&=&\frac{1}{2}(b-a)^{3}\int_{0}^{1}t(1-t)f^{\prime }((1-t)a+tb)dt \\
		&\leq &\frac{1}{2}(b-a)^{3}\int_{0}^{1}t(1-t)[f^{\prime }(a)(1-t)+f^{\prime
		}(b)t]dt \\
		&=&\frac{(b-a)^{3}}{24}(f^{\prime }(a)+f^{\prime }(b)).
	\end{eqnarray*}%
	This proved the right hand side of (\ref{5.1}).
	
	Further, again using integration by parts we have 
	\begin{equation*}
		\int_{a}^{b}\left( x-\frac{a+b}{2}\right) ^{2}f^{\prime }(x)dx=\frac{%
			(b-a)^{2}}{4}(f(b)-f(a))-2\int_{a}^{b}\left( x-\frac{a+b}{2}\right) f(x)dx,
	\end{equation*}%
	and therefore,%
	\begin{equation}
		\int_{a}^{b}\left( x-\frac{a+b}{2}\right) f(x)dx=\frac{(b-a)^{2}}{8}%
		(f(b)-f(a))-\frac{1}{2}\int_{a}^{b}\left( x-\frac{a+b}{2}\right)
		^{2}f^{\prime }(x)dx.  \label{5.2}
	\end{equation}%
	Furhermore, setting $x=(1-t)a+tb$, $\left( x-\frac{a+b}{2}\right)
	^{2}=(b-a)^{2}\left( t-\frac{1}{2}\right) ^{2}$ and $dx=(b-a)dt$, ($0\leq
	t\leq 1$), we get 
	\begin{equation*}
		\int_{a}^{b}\left( x-\frac{a+b}{2}\right) ^{2}f^{\prime
		}(x)dx=(b-a)^{3}\int_{0}^{1}\left( t-\frac{1}{2}\right) ^{2}f^{\prime
		}((1-t)a+tb)dt
	\end{equation*}%
	\begin{eqnarray*}
		&\leq &(b-a)^{3}\int_{0}^{1}\left( t-\frac{1}{2}\right) ^{2}[(1-t)f^{\prime
		}(a)+tf^{\prime }(b)]dt \\
		&=&(b-a)^{3}\left[ f^{\prime }(a)\int_{0}^{1}\left( t-\frac{1}{2}\right)
		^{2}(1-t)dt+f^{\prime }(b)\int_{0}^{1}t\left( t-\frac{1}{2}\right) ^{2}dt%
		\right]  \\
		&=&\frac{(b-a)^{3}}{24}(f^{\prime }(a)+f^{\prime }(b)).
	\end{eqnarray*}%
	Taking into account this in (\ref{5.2}) we obtain the left hand side of
	inequality (\ref{5.1}).
	
	The proof is complete.
\end{proof}

A straightforward calculation shows that the equality in both sides of (\ref%
{5.1}) is attained for $f(x)=k(x^{2}-(a+b)x)+n$, where $k$ and $n$ are
arbitrary real numbers.

\begin{theorem}
	\label{teo5.2}Given $f:[a,b]\rightarrow 
	\mathbb{R}
	$, let $f^{\prime \prime }$ be convex. Then the following inequality holds%
	\begin{equation}
		A\leq \frac{f(a)+f(b)}{2}-\frac{1}{b-a}\int_{a}^{b}f(x)dx\leq B,  \label{5.3}
	\end{equation}%
	where%
	\begin{equation*}
		A=\frac{b-a}{8}(f^{\prime }(b)-f^{\prime }(a))-\frac{(b-a)^{2}}{48}%
		(f^{\prime \prime }(a)+f^{\prime \prime }(b))
	\end{equation*}%
	and%
	\begin{equation*}
		B=\frac{(b-a)^{2}}{24}(f^{\prime \prime }(a)+f^{\prime \prime }(b)).
	\end{equation*}
\end{theorem}

\begin{proof}
	Integration by parts twice gives%
	\begin{equation*}
		\int_{a}^{b}(x-a)(b-x)f^{\prime \prime
		}(x)dx=(b-a)(f(a)+f(b))-2\int_{a}^{b}f(x)dx.
	\end{equation*}%
	Hence,%
	\begin{eqnarray*}
		\frac{f(a)+f(b)}{2}-\frac{1}{b-a}\int_{a}^{b}f(x)dx &=&\frac{1}{2(b-a)}%
		\int_{a}^{b}(x-a)(b-x)f^{\prime \prime }(x)dx \\
		\text{(Set }x &=&(1-t)a+tb\text{, }0\leq t\leq 1\text{)} \\
		&=&\frac{(b-a)^{2}}{2}\int_{0}^{1}t(1-t)f^{\prime \prime }((1-t)a+tb)dt \\
		&\leq &\frac{(b-a)^{2}}{2}\int_{0}^{1}t(1-t)[(1-t)f^{\prime \prime
		}(a)+tf^{\prime \prime }(b)]dt \\
		&=&\frac{(b-a)^{2}}{24}(f^{\prime \prime }(a)+f^{\prime \prime }(b)).
	\end{eqnarray*}%
	The right hand side of (\ref{5.3}) is proved.
	
	Straightforward calculations show that, integration by parts twice yields%
	\begin{eqnarray*}
		&&\int_{a}^{b}\left( x-\frac{a+b}{2}\right) ^{2}f^{\prime \prime }(x)dx \\
		&=&\left( \frac{b-a}{2}\right) ^{2}(f^{\prime }(b)-f^{\prime }(a))-2(b-a)%
		\left[ \frac{f(a)+f(b)}{2}-\frac{1}{b-a}\int_{a}^{b}f(x)dx\right] .
	\end{eqnarray*}%
	Hence,%
	\begin{eqnarray}
		&&\frac{f(a)+f(b)}{2}-\frac{1}{b-a}\int_{a}^{b}f(x)dx  \notag \\
		&=&\frac{b-a}{8}(f^{\prime }(b)-f^{\prime }(a))-\frac{1}{2(b-a)}%
		\int_{a}^{b}\left( x-\frac{a+b}{2}\right) ^{2}f^{\prime \prime }(x)dx.
		\label{5.4}
	\end{eqnarray}%
	Setting $x=(1-t)a+tb$, ($0\leq t\leq 1$) and using the convexity of $%
	f^{\prime \prime }$, we have%
	\begin{equation*}
		\int_{a}^{b}\left( x-\frac{a+b}{2}\right) ^{2}f^{\prime \prime
		}(x)dx=(b-a)^{3}\int_{0}^{1}\left( t-\frac{1}{2}\right) ^{2}f^{\prime \prime
		}((1-t)a+tb)dt
	\end{equation*}%
	\begin{eqnarray*}
		&\leq &(b-a)^{3}\left[ f^{\prime \prime }(a)\int_{0}^{1}\left( t-\frac{1}{2}%
		\right) ^{2}(1-t)dt+f^{\prime \prime }(b)\int_{0}^{1}\left( t-\frac{1}{2}%
		\right) ^{2}tdt\right]  \\
		&=&\frac{(b-a)^{3}}{24}(f^{\prime \prime }(a)+f^{\prime \prime }(b)).
	\end{eqnarray*}%
	By making use of this in (\ref{5.4}) we obtain the left hand side of
	inequality (\ref{5.3}).
	
	The proof is complete.
\end{proof}

It is easy to verify that the equality in both sides of (\ref{5.1}) is
attained for the functions $f(x)=k(2x^{3}-3(a+b)x^{2})+mx+n$, with arbitrary
real numbers $k$, $m$ and $n$.

\begin{remark}
	There are several results in the literature under the condition of the
	convexity of $\left\vert f^{\prime }\right\vert $ or $\left\vert f^{\prime
		\prime }\right\vert $ (see, e.g. \cite{27,28,29}). As far as we know, the
	conditions and assertions of our theorems \ref{teo5.1}, \ref{teo5.2} and \ref%
	{teo5.3}\ completely differ from those known in the literature.
\end{remark}

In the following theorem we give some estimations for the mean value of a
function $f$ whose first derivative is convex.

\begin{theorem}
	\label{teo5.3} Let $f:[a,b]\rightarrow 
	\mathbb{R}
	$ be differentiable and its derivative $f^{\prime }$ be convex. Then
	
	(a) 
	\begin{equation}
		N\leq \frac{1}{b-a}\int_{a}^{b}f(x)dx\leq M,  \label{5.5}
	\end{equation}%
	where%
	\begin{equation*}
		N=\frac{1}{3}(f(a)+2f(b))-\frac{1}{6}f^{\prime }(b)(b-a)
	\end{equation*}%
	and%
	\begin{equation*}
		M=\frac{1}{3}(f(b)+2f(a))+\frac{1}{6}f^{\prime }(a)(b-a);
	\end{equation*}
	
	(b) 
	\begin{equation}
		\mathcal{N}\leq \frac{1}{b-a}\int_{a}^{b}f(x)dx\leq \mathcal{M},  \label{5.6}
	\end{equation}%
	where%
	\begin{equation*}
		\mathcal{N}=f(a)+2f\left( \frac{a+b}{2}\right) -\frac{4}{b-a}\int_{a}^{\frac{%
				a+b}{2}}f(x)dx
	\end{equation*}%
	and%
	\begin{equation*}
		\mathcal{M}=f(b)+2f\left( \frac{a+b}{2}\right) -\frac{4}{b-a}\int_{\frac{a+b%
			}{2}}^{b}f(x)dx.
	\end{equation*}
\end{theorem}

\begin{corollary}
	\begin{equation}
		\int_{\frac{a+b}{2}}^{b}f(x)dx-\int_{a}^{\frac{a+b}{2}}f(x)dx\leq \frac{1}{4}%
		(b-a)(f(b)-f(a)).  \label{5.7}
	\end{equation}
\end{corollary}

\begin{proof}
	Since $f^{\prime }$ is convex, (HH) leads to%
	\begin{eqnarray}
		f^{\prime }\left( \frac{a+x}{2}\right)  &\leq &\frac{1}{x-a}(f(x)-f(a))\leq 
		\frac{f^{\prime }(a)+f^{\prime }(x)}{2};  \label{5.8} \\
		f^{\prime }\left( \frac{x+b}{2}\right)  &\leq &\frac{1}{b-x}(f(b)-f(x))\leq 
		\frac{f^{\prime }(x)+f^{\prime }(b)}{2}.  \label{5.9}
	\end{eqnarray}%
	Multiplying the inequalities (\ref{5.8}) by $(x-a)$ and integrating over $%
	[a,b]$, after simple calculations we obtain%
	\begin{equation*}
		2(b-a)f\left( \frac{a+b}{2}\right) -4\int_{a}^{\frac{a+b}{2}}f(x)dx\leq
		\int_{a}^{b}f(x)dx-f(a)(b-a)
	\end{equation*}%
	\begin{equation*}
		\leq \frac{1}{4}f^{\prime }(a)(b-a)^{2}+\frac{1}{2}(b-a)f(b)-\frac{1}{2}%
		\int_{a}^{b}f(x)dx.
	\end{equation*}%
	The above inequalities can be written as two seperate inequalities:%
	\begin{equation}
		\frac{1}{b-a}\int_{a}^{b}f(x)dx\leq \frac{1}{3}(2f(a)+f(b))+\frac{1}{6}%
		f^{\prime }(a)(b-a)  \label{5.10}
	\end{equation}%
	and%
	\begin{equation}
		\frac{1}{b-a}\int_{a}^{b}f(x)dx+\frac{4}{b-a}\int_{a}^{\frac{a+b}{2}%
		}f(x)dx\geq f(a)+2f\left( \frac{a+b}{2}\right) .  \label{5.11}
	\end{equation}%
	In a similar way, multiplying inequalities (\ref{5.9}) by $(b-x)$ and
	integrating over $[a,b]$, after some calculations we have the following two
	inequalities:%
	\begin{equation}
		\frac{1}{b-a}\int_{a}^{b}f(x)dx\geq \frac{1}{3}(f(a)+2f(b))-\frac{1}{6}%
		f^{\prime }(b)(b-a)  \label{5.12}
	\end{equation}%
	and%
	\begin{equation}
		\frac{1}{b-a}\int_{a}^{b}f(x)dx+\frac{4}{b-a}\int_{\frac{a+b}{2}%
		}^{b}f(x)dx\leq f(b)+2f\left( \frac{a+b}{2}\right) .  \label{5.13}
	\end{equation}%
	Now, the inequalities (\ref{5.10}) and (\ref{5.12}) yields (\ref{5.5}) and
	the inequalities (\ref{5.11}) and (\ref{5.13}) yields (\ref{5.6}). The
	Corollary follows by subtracting (\ref{5.11}) from (\ref{5.13}).
	
	The proof is complete.
\end{proof}

\begin{example}
	For $0<a<x<b<\infty $ and $f(x)=\ln x$, the inequality (\ref{5.7}) yields%
	\begin{equation}
		a^{\frac{3a+b}{4(a+b)}}\cdot b^{\frac{a+3b}{4(a+b)}}\leq \frac{a+b}{2}.
		\label{5.14}
	\end{equation}%
	Since $\alpha +\beta =1$ for $\alpha =\frac{3a+b}{4(a+b)}$ and $\beta =\frac{%
		a+3b}{4(a+b)}$, then by the generalized AM-GM inequality we have%
	\begin{equation}
		a^{\alpha }\cdot b^{\beta }<\alpha \cdot a+\beta \cdot b=\frac{3a+b}{4(a+b)}%
		\cdot a+\frac{a+3b}{4(a+b)}\cdot b.  \label{5.15}
	\end{equation}%
	A simple calculation shows that%
	\begin{equation*}
		\frac{a+b}{2}<\frac{3a+b}{4(a+b)}\cdot a+\frac{a+3b}{4(a+b)}\cdot b,
	\end{equation*}%
	and therefore, the inequality (\ref{5.14}) is better than (\ref{5.15}).
\end{example}


%
%



\end{document}